\documentclass{article}
\usepackage{amssymb}
\usepackage[all]{xy}
\title{The Computation of the Logarithmic Cohomology for Plane Curves}
\author{Francisco Jes\'{u}s Castro-Jim\'{e}nez and Nobuki Takayama}

\date{November 28, 2007, Revised December 13, 2007}

\newenvironment{proof}{\par\noindent Proof:\ }{${\tt [}\kern-0.2mm{\tt ]}$}
\newtheorem{theorem}{Theorem}[section]
\newtheorem{proposition}[theorem]{Proposition}
\newtheorem{example}[theorem]{Example}
\newtheorem{algorithm}[theorem]{Algorithm}

\newcommand{\cV}{{\mathcal V}}
\newcommand{\CC}{{\mathbb C}}
\newcommand{\NN}{{\mathbb N}}
\newcommand{\PP}{{\mathbb P}}
\newcommand{\depth}{{\rm depth}}

\begin{document}
\maketitle

\noindent
Abstract:
We will give algorithms of computing bases of logarithmic cohomology
groups for square-free polynomials in two variables.

\section{Introduction}
Let us denote by $R=\CC[x]=\CC[x_1,\ldots,x_n]$ the polynomial ring,
by  $A_n={\bf C}\langle
x_1,\ldots,x_n,\partial_1,\ldots,\partial_n\rangle$ the Weyl algebra
of order $n$ over the complex numbers $\CC$ and by
$(\Omega_R^\bullet,d)$ the complex of polynomial (or regular)
differential forms (i.e. the complex of differential forms with
polynomial coefficients) where $d$ is the exterior derivative.

The elements of $A_n$ are called linear differential operators with
polynomial coefficients. An element $P(x,\partial)$ in $A_n$ can be
written as a finite sum $P(x,\partial) = \sum_\alpha a_\alpha(x)
\partial^\alpha$ where $\alpha=(\alpha_1,\ldots,\alpha_n)\in
\NN^n$, $a_\alpha(x)\in R$ and $\partial^\alpha =
\partial_1^{\alpha_1}\cdots \partial_n^{\alpha_n}$. Here
$\partial_i$ stands for the partial derivative
$\frac{\partial}{\partial x_i}$.

For a non zero polynomial $f\in R$ we denote by $R_f$ the ring of
rational functions $$R_f = \{\frac{g}{f^m} \, \vert \, g\in R ,
m\in \NN\}$$ and by $(\Omega^\bullet_f,d):=(R_f \otimes_ R
\Omega^\bullet_R,d)$ the complex of rational differential forms
with coefficients in $R_f$ where $d$ is the corresponding exterior
derivative.

Let us denote by $Der_\CC(R)$ the free $R$--module of polynomial
vector fields (or equivalently of $\CC$-linear derivations of $R$).
Following K. Saito \cite{Saito80} we will denote by $Der_R(-\log f)$
the $R$--module of logarithmic vector fields with respect to $f$,
i.e.
$$Der_R(-\log f) = \{\delta = \sum_{i=1}^n a_i(x)
\partial_i \in Der_\CC(R) \, \vert \, \delta(f) \in R\cdot f\}.$$ $Der(-\log f)$
is canonically isomorphic to the $R$--module
$Syz_R(\partial_1(f),\ldots,\partial_n(f),f)$ of syzygies among
$(\partial_1(f),\ldots,\partial_n(f),f)$.  This isomorphism
associates the logarithmic vector field $\delta=\sum_i a_i(x)
\partial_i $ with  the syzygy
$(a_1(x),\ldots,a_n(x),-\frac{\delta(f)}{f}).$ We will denote simply
$Der(-\log f)$ if no confusion is possible.

If $f$ is a non zero constant, then $Der(-\log f)=Der_\CC(R)$. So we
will assume from now that $f$ is a non constant polynomial in $R$.

It is clear that $$fDer_\CC(R) \subset Der_R(-\log f) \subset
Der_\CC(R)$$ and then $Der(-\log f)$ has rank $n$ as $R$--module.
The $R$--module $Der_R(-\log f)$ does not depend on the polynomial
$f$ but only on the hypersurface $D=\cV(f):=\{a\in \CC^n\, \vert\,
f(a)=0\} \subset \CC^n$.

Assume $f$ is reduced (i.e. $f$ is square-free). According to K.Saito
\cite{Saito80} a rational differential $p$-form $\omega \in
\Omega^p_f$ is said to be logarithmic with respect to $f$ (or with
respect to the hypersurface $D=\cV(f)\subset \CC^n$) if both $f\omega$
and $fd\omega$ are regular (i.e. $f\omega\in \Omega_R^p$ and
$fd\omega \in \Omega^{p+1}_R$). 
We denote by $\Omega^p(\log f)$ the
$R$--module of logarithmic differential $p$--forms with respect to
$f$.
K. Saito \cite[Corollary 1.6]{Saito80} proved that $Der_R(-\log f)$
is a reflexive $R$--module whose dual is $\Omega^1(\log f)$.
We denote by $(\Omega^\bullet(\log f), d)$ the complex
$$ 0 \longrightarrow               \Omega^0(\log f) 
     \stackrel{d}{\longrightarrow} \Omega^1(\log f) 
     \stackrel{d}{\longrightarrow} 
      \cdots
     \stackrel{d}{\longrightarrow} \Omega^n(\log f) \longrightarrow 0
$$
which will be called the logarithmic de Rham complex and
is also, for simple notation, 
denoted by $\Omega^\bullet(\log f)$ if no confusion arises.
    
Algorithms of computing dimensions and bases of the de Rham cohomology groups
$H^i(\Omega_f^\bullet)$ are given by
T.Oaku and N.Takayama \cite{oaku-takayama-1999}, \cite{oaku-takayama-2001}
and U.Walther \cite{walther}.
Here, $f$ is any non-zero polynomial in $n$-variables.
The purpose of this paper is to give algorithms
of computing dimensions and bases of the logarithmic de Rham cohomology groups
$H^i(\Omega^\bullet(\log f))$ as $\CC$-vector spaces
in the case of two variables.

\subsection{Logarithmic Comparison Theorem}

The rings $R$ and $R_f$ have natural structures of left
$A_n$--module where $\partial_i$ acts on a polynomial $g$ and on a
rational function $\frac{g}{f^m}$ as the partial derivative with
respect to $x_i$.

The de Rham complex of a left $A_n$--module $M$, denote by
$DR(M)$, is by definition the complex of $\CC$--vector spaces
$(M\otimes _R \Omega_R^\bullet,\nabla^\bullet)$ where $$\nabla^p :
M\otimes_R \Omega_R^p \rightarrow M\otimes_R \Omega^{p+1}_R$$ is
defined, for $p\geq 1$, by $\nabla^p(m\otimes \omega) =
\nabla^0(m)\wedge \omega + m\otimes d\omega$ and $\nabla^0(m)=
\sum_i \partial_i(m) \otimes dx_i$. 
Note that $am \otimes \omega = m \otimes a \omega$
for $m \in M$, $\omega \in \Omega^p$ and $a \in R$.
The complexes $\Omega^\bullet
_f $ and $DR(R_f)$ are naturally isomorphic.

For any non zero $f\in R$, the inclusion $i_f$ is a natural
morphism of complexes
$$i_f : \Omega^\bullet (\log f) \rightarrow \Omega^\bullet_f.$$
We say (see \cite{Castro-Mond-Narvaez-96}) that $f$ satisfies the
(global)  logarithmic Comparison Theorem if the morphism $i_f$ is
a quasi-isomorphism (i.e. if $i_f$ induces an isomorphism
$H^p(\Omega^\bullet(\log f)) \rightarrow H^p(\Omega^\bullet_f)$
for any $p$).

If $n=2$, by \cite[Cor. 2.7]{Castro-Mond-Narvaez-96} and \cite[Th.
1.3]{cal-mon-nar-cas}, $i_f$ is a quasi-isomorphism if and only if
$f$ is a quasi-homogeneous polynomial.

\subsection{The case $n=2$. Bases for $Der_R(-\log f)$}\label{casen=2}
If $n=2$, any reflexive $R$--module is projective and then, by
Quillen-Suslin theorem, this $R$--module is free. So, if $n=2$, the
$R$-module $Der_R(-\log f)$ is free of rank 2. In this case, we
would like to compute a basis of $Der_R(\log f)$ by taking the
polynomial $f=f(x,y)$ as input. 
By using the isomorphism
$$Der(-\log f) \simeq Sys_R(\partial_1(f),\partial_2(f),f)$$ and
using Groebner basis computation,
a system of generators of
$Der_R(-\log f)$ can be calculated. Then we can apply Quillen-Suslin
algorithm (as presented for example in \cite{logar-sturmfels} and
implemented in \cite{fabianska}) to compute such a basis. Known
Quillen-Suslin algorithms use Groebner bases computation.
Nevertheless, in some cases, for a big family of polynomials
$f(x_1,x_2)$ we will use an easier way to compute a basis of
$Der(-\log f)$.

First of all, we can assume $f$ to be a reduced polynomial since
$Der(-\log f)$ depends only on the affine plane curve
$D=\cV(f)=\{(a_1,a_2)\in \CC^2\, \vert\, f(a_1,a_2)=0\} \subset
\CC^2$.

Assume the plane curve $D=\cV(f)$ is not smooth. The singular points
of the plane curve $D=\cV(f)$ (i.e. the affine algebraic set
$$Sing(D):=\cV(f,f_1,f_2)
=\{\underline{a}=(a_1,a_2)\in \CC^2\, \vert \,
f(\underline{a})=f_1(\underline{a})=f_2(\underline{a})=0\})$$
--where $f_1=\partial_1(f),\, f_2=\partial_2(f) $ -- consists in a
finite number of points (and it is not the empty set).

We will consider the affine plane $\CC^2 $ as a Zariski open subset
of the  projective plane $\PP_2(\CC)$, the affine point $(a_1,a_2)$
is mapped into the point with homogeneous coordinates $(1:a_1:a_2)$.
Coordinates in $\PP_2(\CC)$ will be denoted by $(x_0:x_1:x_2)$ and then
the line at infinity is defined by $x_0=0$.

Let us denote $h=H(f), h_1=H(f_1)$ and $h_2=H(f_2)$ where $H(-)$
denotes dehomogenization with respect to the variable $x_0$. We
will denote by $Z=\cV_\PP(h,h_1,h_2)\subset \PP_2(\CC)$ (resp.
$Z'=\cV(h,h_1,h_2) \subset \CC^3$ ) the projective algebraic set
(resp. the affine algebraic set) defined by the polynomials
$h,h_1,h_2$. The non-empty set $Z$ (resp. $Z'$) consists of a finite
number of points  in $\PP_2(\CC)$ (resp. a finite number of
straight lines in $\CC^3$). Denote by $S=\CC[x_0,x_1,x_2]$ the  ring
of polynomials graded by the degree of the polynomials. If
$J=(h,h_1,h_2)$ denotes the ideal in $S$ generated by $h,h_1,h_2$
then the quotient ring $S/J$ has Krull dimension 1. Let us denote by
$S_+$ the irrelevant ideal in $S$, i.e. the ideal generated by
$x_0,x_1,x_2$.

\begin{proposition} The graded ring $S/J$ is Cohen-Macaulay if and
only if $J$ is unmixed (i.e. $S_+$ is not an embedded prime associated
with $J$).
\end{proposition}

\begin{proof} If $S/J$ is Cohen-Macaulay then $J$ is unmixed (see
\cite{matsumura}). If $J$ is unmixed then $S_+$ is not an embedded
prime of $J$ and then the set of non zero-divisors of $S/J$ contains
homogeneous elements of positive degree. That proves $depth(S/J)\geq
1$ but we also have $\depth(S/J)\leq \dim(S/J) = 1$.
\end{proof}

If $S/J$ is Cohen-Macaulay then the projective dimension of $S/J$ is
2 and $J$ satisfies the Hilbert-Burch Theorem \cite{eisenbud}, i.e.
there exists an exact sequence $$0 \rightarrow S^{2}
\stackrel{\phi_2}{\longrightarrow} S^3
\stackrel{\phi_1}{\longrightarrow} J \rightarrow 0
$$ where $\phi_1(g_0,g_1,g_2)=g_0h + g_1 h_1 + g_2 h_2$ and $\phi_2$
is defined by a syzygy matrix of $\phi_1$. In particular, since
$\ker(\phi_1)=Syz_S(h,h_1,h_2)$ is a graded free $S$--module of rank
2 we can compute $\{ s^{(1)}=(s_{10},s_{11},s_{12}),
s^{(2)}=(s_{20},s_{21},s_{22}) \}$ a minimal system of generators
and this system is in fact a basis of $\ker(\phi_1)$. By
dehomogenization (i.e. by setting $x_0=1$), we obtain a system
$\{s^{(1)}_{\vert x_0=1}, s^{(2)}_{\vert x_0=1} \}$ of generators
of $Syz_R(f,f_1,f_2) \simeq Der_R(-\log f)$ and since this
$R$--module is free of rank 2, this last system is in fact a basis.

If $S/J$ is not Cohen-Macaulay we cannot apply, in general,
the Hilbert-Burch theorem and the previous procedure fails to compute a
basis of $Der_R(-\log f)$.

\begin{example} \rm \label{example:qs}
(a) Consider the polynomial $f=(x^3+y^4+
xy^3)(x^2-y^2)$. With the notations as before (and writing $x_1=x,
x_2=y, x_0=t)$ we can use {\tt Macaulay 2} to prove that the corresponding $S/J$ is
Cohen-Macaulay and to compute a minimal system of generators of $Syz_S(h,h_1,h_2)$ and then a basis
of $Der_R(-\log f)$.

{\footnotesize
\begin{verbatim}
Macaulay 2, version 0.9.2
--Copyright 1993-2001, D. R. Grayson and M. E. Stillman
--Singular-Factory 1.3b, copyright 1993-2001, G.-M. Greuel, et al.
--Singular-Libfac 0.3.2, copyright 1996-2001, M. Messollen

i1 : R=QQ[t,x,y];

i2 : f=(x^3+y^4+x*y^3)*(x^2-y^2);

i3 : f1=diff(x,f),f2=diff(y,f),h=homogenize(f,t),h1=homogenize(f1,t),h2=homogenize(f2,t);

i4 : Jf=ideal(h,h1,h2);

o4 : Ideal of R

i5 : pdim coker gens Jf

o5 = 2

i6 : Syzf=kernel matrix({{h1,h2,h}})

o6 = image {5} | x3+1/3x2y-4/3xy2 -tx2+4txy+3x2y+4xy2-y3 |
           {5} | 2/3x2y+1/3xy2-y3 tx2-txy+3ty2+2xy2+4y3  |
           {6} | -5x2-5/3xy+6y2   5tx-18ty-15xy-23y2     |

                             3
o6 : R-module, submodule of R

i7 : mingens Syzf

o7 = {5} | x3+1/3x2y-4/3xy2 -tx2+4txy+3x2y+4xy2-y3 |
     {5} | 2/3x2y+1/3xy2-y3 tx2-txy+3ty2+2xy2+4y3  |
     {6} | -5x2-5/3xy+6y2   5tx-18ty-15xy-23y2     |

             3       2
o7 : Matrix R  <--- R
\end{verbatim}
}

Then the basis of $Der_R(-\log f)$ is
{\footnotesize
$$\{(x^3+\frac{1}{3}x^2y-\frac{4}{3}xy^2)\partial_x + ( \frac{2}{3}x^2y+\frac{1}{3}xy^2-y^3) \partial_y
, \, (-x^2+4xy+3x^2y+4xy^2-y^3)\partial_x + (x^2-xy+3y^2+2xy^2+4y^3) \partial_y\} $$}

\noindent (b) Consider the polynomial $g=(x^3+y^4+
xy^3)(x^2+y^2)$. With the notations as before (and writing $x_1=x,
x_2=y, x_0=t)$ we can use {\tt Macaulay 2} to prove that the corresponding $S/J$ is not
Cohen-Macaulay and the minimal number of generators of $Syz_S(h,h_1,h_2)$ is 3. We can continue
the last {\tt Macaulay 2} session:

{\footnotesize
\begin{verbatim}

i8 : g=(x^3+y^4+x*y^3)*(x^2+y^2);

i9 : g1=diff(x,g),g2=diff(y,g),h=homogenize(g,t),h1=homogenize(g1,t),h2=homogenize(g2,t);

i10 : Jg=ideal(h,h1,h2);

i11 : pdim coker gens Jf

o11 = 3

i12 : Syzg=kernel matrix({{h1,h2,h}})

o12 =
image
{5} | tx2-5x3-4txy-20/3x2y-2xy2-5/3y3 x4+4/3x3y+x2y2+4/3xy3    tx3-tx2y+4x3y+4txy2+16/3x2y2+2xy3+4/3y4 |
{5} | tx2+txy-10/3x2y-3ty2-5xy2-1/3y3 2/3x3y+x2y2+2/3xy3+y4    -txy2+8/3x2y2+3ty3+4xy3+2/3y4           |
{6} | -5tx+25x2+18ty+100/3xy+11/3y2   -5x3-20/3x2y-13/3xy2-6y3 -5tx2+5txy-20x2y-18ty2-80/3xy2-16/3y3   |

                              3
o12 : R-module, submodule of R

i13 : mingens Syzg

o13 =
{5} | tx2-5x3-4txy-20/3x2y-2xy2-5/3y3 x4+4/3x3y+x2y2+4/3xy3    tx3-tx2y+4x3y+4txy2+16/3x2y2+2xy3+4/3y4 |
{5} | tx2+txy-10/3x2y-3ty2-5xy2-1/3y3 2/3x3y+x2y2+2/3xy3+y4    -txy2+8/3x2y2+3ty3+4xy3+2/3y4           |
{6} | -5tx+25x2+18ty+100/3xy+11/3y2   -5x3-20/3x2y-13/3xy2-6y3 -5tx2+5txy-20x2y-18ty2-80/3xy2-16/3y3   |

              3       3
o13 : Matrix R  <--- R
\end{verbatim}
}
We will revisit this example in Example \ref{example:small}.
\end{example}

\section{Logarithmic $A_n$--modules}

Let us denote by $M^{\log f}$ the quotient $A_n$--module $M^{\log
f } = \frac{A_n}{A_n Der_R(-\log f)}$. Moreover, we denote by
$\widetilde{Der}_R(-\log f)$ the set $$\widetilde{Der}_R(-\log f)
= \{\delta + \frac{\delta(f)}{f}\, \vert \, \delta \in Der_R(-\log
f) \}$$ and by $\widetilde{M}^{\log f}$ the quotient $A_n$--module
$$ \widetilde{M}^{\log f} = \frac{A_n}{A_n \widetilde{Der}_R(-\log
f)}.$$

As quoted in subsection \ref{casen=2}, for $n=2$ the $R$--module
$Der(-\log f)$ (and hence $\Omega^1(\log f)$) is free of rank 2.
Moreover, by \cite[1.8]{Saito80} there exists a $R$-basis
$\{\delta_1,\delta_2\}$ of $Der(-\log f)$ satisfying $\det(A)=f$
where
$$\delta_i=a_{i1}\partial_1 + a_{i2}\partial_2, \, \, \, i=1,2$$  and $A$
is the matrix $(a_{ij})$. Then the dual basis of
$\{\delta_1,\delta_2\}$ is $\{\omega_1,\omega_2\}$ with $$\omega_1 =
\frac{1}{f}(a_{22}dx_1 - a_{21}dx_2)\, \, \,  \omega_2 =
\frac{1}{f}(-a_{12}dx_1+a_{11}dx_2).$$

The $R$--module $\Omega^2(\log f)$ is free of rank 1 and $\omega_1
\wedge \omega_2$ is a basis of it. Moreover we have $\omega_1 \wedge
\omega_2 = \frac{dx_1\wedge dx_2}{f}$.

\begin{proposition}\label{omega-sol} Let $f\in R=\CC[x,y]$ be  a non zero reduced polynomial.
There exists  a   natural  quasi-isomorphism
 $$\Omega^\bullet(\log f)
\stackrel{\simeq}{\longrightarrow } {\bf R} Hom_{A_2}(M^{\log f},
R)$$ where the last complex is the solution complex  of $M^{\log f}$
with values in $R$.
\end{proposition}

This Proposition is proven in \cite{Calderon-thesis} in a more
general setting using the notion of $V_0$--module. We will give
here a direct proof to apply for our algorithm of computing logarithmic 
cohomology groups.

\begin{proof} F.J. Calder\'{o}n \cite{Calderon-thesis} defines the so
called {\em logarithmic Spencer complex} associated with $M^{\log
f}$. In our situation, once a basis $\{\delta_1,\delta_2\}$ is
fixed in $Der(-\log f)$, this complex is nothing but
\begin{equation}\label{log-spencer} 0 \rightarrow A
\stackrel{\epsilon_2}{\longrightarrow} A^2 \stackrel{\epsilon_1
}{\longrightarrow} A \rightarrow 0
\end{equation} where $A$ stands for $A_2$, the $A$-module morphism
$\epsilon_1$ is defined by $\epsilon_1(P_1,P_2)=P_1\delta_1 +P_2
\delta_2$ (for $P_i\in A$) and $\epsilon_2$ is defined by
$\epsilon_2(Q)=Q(-\delta_2-b_1,\delta_1-b_2)$ for $Q\in A$ and the
polynomials  $b_i$ being defined by the equality
$[\delta_1,\delta_2] = \delta_1\delta_2-\delta_2\delta_1 =
b_1\delta_1 + b_2\delta_2$. In \cite{Calderon-thesis} it is proven
that this complex is a $A$--free resolution of the module $M^{\log
f}$. We will use this resolution to find a complex of
$\CC$--vector spaces representing the solution complex ${\bf R}
Hom_A (M^{\log f},R)$. Applying the functor $Hom_A(-,R)$ to the
logarithmic Spencer complex and using the natural isomorphism
$R\simeq Hom_A(A,R)$, we obtain the complex
$$0 \rightarrow R \stackrel{\epsilon_1^*}{\longrightarrow} R^2
\stackrel{\epsilon_2^*}{\longrightarrow } R \rightarrow 0$$ where
$\epsilon_1^*(g)= (\delta_1(g),\delta_2(g))$ for $g\in R$ and
$\epsilon_2^*(h_1,h_2)=\delta_1(h_2)-\delta_2(h_1)-b_1h_1-b_2h_2$
for $h_i \in R$. There is a natural morphism of complexes

$$ \begin{array}{ccccc}  \Omega^0(\log f) = R  & \stackrel{d}{\rightarrow}  & \Omega^1(\log f) & \stackrel{d}{\rightarrow} &  \Omega^2(\log
f)\\ &&&& \\{}^{\eta_0}\downarrow & & {}^{\eta_1} \downarrow& & {}^{\eta_2}\downarrow\\
R & \stackrel{\epsilon_1^*}{\longrightarrow} & R^2 &
\stackrel{\epsilon_2^*}{\longrightarrow }&  R
\end{array}
$$ where $\eta_0=id$,  $\eta_1(h_1 \omega_1 + h_2 \omega_2)= (h_1,h_2)$ and
$\eta_2(g \omega_1 \wedge \omega_2) = g$ for $h_1,h_2,g \in R$ and
where $\{\omega_1,\omega_2\}$ is the dual basis in $\Omega^1(\log
f)$ of the basis $\{\delta_1,\delta_2\}$ in $Der(-\log f)$. It is
obvious that this morphism $\eta_\bullet$ of complexes of vector
spaces is in fact an isomorphism of complexes. That proves the
proposition.
\end{proof}

To each  finitely generated left $A_n$--module $M$ we associate
the complex of finitely generated right $A_n$--modules ${\bf R}
Hom_{A_n}(M,A_n)$. To this one we associate the complex of
finitely generated left $A_n$--modules $Hom_R(\Omega^n_R, {\bf R}
Hom_{A_n}(M,A_n))$ which is by definition the dual $M^*$ of the
left $A_n$--module $M$.

If $M$ is holonomic (i.e. if the dimension of the characteristic
variety of $M$ is $n$) then it can be shown that $Ext^i
_{A_n}(M,A_n)=0$ for $i\not=n$  and then $M^*$ is the left
holonomic $A_n$--module $Hom_R(\Omega^n_R,Ext^n_{A_n}(M,A_{n}))$
(see e.g. \cite[pag. 41]{Meb-hermann-89}). Assume
$Ext^n_{A_n}(M,A_n) = \frac{A_n}{J}$ for some right ideal
$J\subset A_n$. Then $Hom_R(\Omega^n_R,A_n/J)$ is naturally
isomorphic to the left $A_n$--module $\frac{A_n}{J^T}$ where $J^T$
is the left ideal $J^T=\{P^T \, \vert\, P\in J\}$ and $P^T$ is the
formal adjoint of the operator $P$.

If $N_1,N_2$ are finitely generated left $A_n$--modules there
exists a natural isomorphism of complexes $${\bf R}
Hom_{A_n}(N_1,N_2) {\rightarrow} {\bf R}Hom_{A_n}({\bf R}
Hom_{A_n}(N_2,A_n),{\bf R} Hom_{A_n}(N_1,A_n))$$ and then a
natural isomorphism $${\bf R} Hom_{A_n}(N_1,N_2) {\rightarrow}
{\bf R}Hom_{A_n}(N_2^*,N_1^*).$$

In particular, if $N_2=R=\CC[x_1,\ldots,x_n]$ then there exists a
natural isomorphism from ${\bf R} Hom_{A_n}(N_1,R)$ (i.e. the
solution complex of $N_1$) to $$ {\bf R}Hom_{A_n}(R^*,N_1^*).$$ As
the complex ${\bf R} Hom_{A_n}(R,A_n)$ is naturally isomorphic to
$\Omega_R^n$ we can identify $R$ and $R^*$ and then we have a
natural isomorphism \begin{equation}\label{sol-derham} {\bf R}
Hom_{A_n}(N_1,R) \stackrel{\simeq}{\rightarrow} {\bf
R}Hom_{A_n}(R,N_1^*) \stackrel{\simeq}{\rightarrow} DR(N_1^*).
\end{equation}

\begin{proposition}\label{dual} Let $f\in\CC[x,y]$ be a non zero  reduced
polynomial. Then there exists a natural isomorphism  $$(M^{\log
f})^* \simeq \widetilde{M}^{\log f}.$$
\end{proposition}

\begin{proof} This is one of the main results  in \cite{castro-ucha-jsc}. We
include here its proof   for the sake of  completeness. First of
all,  both $A_2$--modules $M^{\log f}$ and $ \widetilde{M}^{\log f}$
are holonomic. That can be deduced from \cite[Cor.
4.2.2]{Calderon-thesis} since the set of principal symbols
$\{\sigma(\delta_1), \sigma(\delta_2)\}$ is a regular sequence in
the polynomial ring $R[\xi_1,\xi_2]$ and then the Krull dimension of
the quotient ring $$R[\xi_1,\xi_2]/\langle \sigma(\delta_1),
\sigma(\delta_2) \rangle$$ is 2. Then the characteristic variety of
both $A_2$--modules $M^{\log f}$ and $ \widetilde{M}^{\log f}$ has
dimension 2 and the modules are holonomic.

We will use the logarithmic Spencer complex associated with $M^{\log
f}$ (see the complex (\ref{log-spencer})) in order to compute
$Ext^2_{A}(M,A)$ where $A=A_2$ and $M=M^{\log f}$. Applying the
functor $Hom_{A_2}(-,A_2)$ to the complex (\ref{log-spencer}) we get
(by using the natural isomorphism $Hom_{A_2}(A_2,A_2)\simeq A_2$)
$$ 0 \longrightarrow A \stackrel{\overline{\epsilon_1}}{\longrightarrow}
A^2 \stackrel{\overline{\epsilon_2}}{\longrightarrow } A
\longrightarrow 0$$ where
$\overline{\epsilon_1}(P)=(\delta_1P,\delta_2P)$ and
$\overline{\epsilon_2}(P_1,P_2)=(-\delta_2-b_1)P_1 +
(\delta_1-b_2)P_2$. Then we have $$Ext_{A}^2(M,A) \simeq
\frac{A}{(-\delta_2-b_1,\delta_1-b_2)A}.$$ So, $$M^* \simeq
\frac{A}{A((-\delta_2-b_1)^T,(\delta_1-b_2)^T)}.$$ Finally,
$(-\delta_2-b_1)^T = \delta_2 + \frac{\delta_2(f)}{f}$ and $
(\delta_1-b_2)^T = -\delta_1 -\frac{\delta_1(f)}{f}$ (see \cite[Cor.
3.1]{castro-ucha-jsc}).
\end{proof}

\begin{theorem}\label{main}
For any non zero reduced polynomial $f\in \CC[x,y]$,
 the complexes $\Omega^\bullet (\log f) $ and $DR(\widetilde{M}^{\log f})$ are
naturally quasi-isomorphic.
\end{theorem}

As a consequence of this theorem and by 
\cite{oaku-takayama-1999}, \cite{oaku-takayama-2001} and \cite{walther},
the cohomology of the complex $\Omega ^\bullet(\log f)$ can be
computed starting with the given polynomial $f$, since a system of
generators of the $R$-module $\widetilde{Der}_R(-\log f)$ can be
computed using the $R$--syzygies of
$(\partial_1(f),\partial_2(f),f)$.

\begin{proof} Let us simply denote $R=\CC[x,y]$, $A=A_2$, $M=M^{\log f}$,
$\widetilde{M}=\widetilde{M}^{\log f}$.

By Proposition \ref{omega-sol} there exists a natural isomorphism
$$\Omega^\bullet(\log f) \stackrel{\simeq}{\longrightarrow} {\bf R}
Hom_{A}(M,R)$$ and by equation (\ref{sol-derham}) there exists a
natural isomorphism
$${\bf R} Hom_{A}(M,R) \stackrel{\simeq }{\longrightarrow}
DR(M^*).$$ By Proposition \ref{dual} we have $DR(M^*) \simeq
DR(\widetilde{M})$.

We can give the explicit form of this  quasi-isomorphism of
complexes $\tau^\bullet : \Omega^\bullet(\log f) \rightarrow
DR(\widetilde{M})$.

$\tau^0 : R \rightarrow \widetilde{M}$ is defined by
$\tau^0(g)=\overline{gf}$ where $\overline{\,(\,)\,}$ means the
equivalent class in the ideal $A_2\widetilde{Der}_R(\log f)$.

$\tau^1 : \Omega^1(\log f) \rightarrow \widetilde{M}\otimes_R
\Omega_R^1$ is defined by $$\tau^1(c_1\omega_1+c_2\omega_2)= \sum_i
\overline{c_i}\otimes f\omega_i.$$

$\tau^2 : \Omega^2(\log f) \rightarrow \widetilde{M}\otimes_R
\Omega^2_R$ is defined by $\tau^2(g \omega_1\wedge \omega_2) =
\overline{g}\otimes f\omega_1 \wedge \omega_2$.
\end{proof}

\section{Algorithm}
Let us summarize our algorithm of computing logarithmic 
cohomology groups in the two dimensional case.
Most tensor products $\otimes$ in the sequel are over $A_2$.
If we omit the subscript $A_2$ for $\otimes$,
it means that the tensor product is over $A_2$.

\begin{algorithm} \rm \ \\  \label{algorithm:main}
Input: a non zero reduced polynomial $f(x,y)$ \\
Output: dimensions and bases of $H^i(\Omega^\bullet(\log f))$. \\ 
\begin{enumerate}
\item Compute a free basis $s=(s_0,s_1,s_2)$ and $t=(t_0,t_1,t_2)$
of the syzygy module of $f, f_x, f_y$ over the polynomial ring ${\bf C}[x,y]$.
This step can be performed by the following way.
\begin{enumerate}
\item Compute the minimal syzygy 
of $h(f)$, $h(f_x)$, $h(f_y)$.
Here, $h(g)$ is the homogenization of $g$.
If the number of generators is $2$, then the dehomogenizations of these generators are
$s$ and $t$.
\item If we fail on the first step, apply an algorithm for the Quillen-Suslin theorem to obtain $s$ and $t$ (call the procedure Quillen-Suslin). 
\end{enumerate}
\item Define a left ideal in $A_2$ by
\begin{equation}  \label{eq:tMlogf}
I = A_2 \cdot \left\{ 
  -s_0+s_1\partial_x+s_2\partial_y,  
  -t_0+t_1\partial_x+t_2\partial_y \right\}.
\end{equation}   
Compute the dimensions and bases of the de Rham cohomology groups
for ${\widetilde M}= A_2/I$
with the algorithm in \cite{oaku-takayama-1999}, \cite{oaku-takayama-2001}.
In other words, 
replace the $A_2$-module ${\bf C}[x,y,1/f]$ by $A_2/I$ 
of (\ref{eq:tMlogf}) in the algorithm 1.2 in \cite{oaku-takayama-1999}.
\item 
The bases of the previous step are given in
$A_2/(\partial_x A_2 + \partial_y A_2) \otimes {\widetilde M}^\bullet$
where ${\widetilde M}^\bullet$ is $(1,1,-1,-1)$-adaptive free resolution of $\widetilde M$.
Bases of de Rham  cohomology groups in 
$\Omega^\bullet \otimes \widetilde M \simeq_{q.i.s} DR(\widetilde M) 
                        \simeq_{q.i.s} \Omega^\bullet(\log f) $
are determined by 
the transfer algorithm of U.Walther \cite[Theorem 2.5 (Transfer Theorem)]{walther}
and the correspondence $\tau^i$ given in our Theorem \ref{main}.
Here, $\Omega^\bullet$ is the Koszul resolution of the right $A_2$-module
$A_2/(\partial_x A_2 + \partial_y A_2)$.
\end{enumerate}
\end{algorithm}

In the first step, we  should firstly try to find the minimal syzygy.
Because, mostly it is faster than applying implementations and algorithms 
for the Quillen-Suslin theorem.

The following example will illustrate how our algorithm works.
\begin{example} \rm   
\def\AC{
  \pmatrix{
    A_2    \cr
    \oplus \cr
    A_2    \cr
  }
}

We consider the case of $f=xy(x-y)$.
Two canonical generators of $I=\widetilde{Der}_R(\log f)$ are
$$\ell_1 = 3 + x \partial_x + y \partial_y, \ 
  \ell_2 = -(2x-y)+(-x^2+xy)\partial_x 
$$  
The associated canonical logarithmic forms are  
$$ \omega_1 = \frac{1}{f} x (x-y) dy, \ 
   \omega_2 = \frac{1}{f}(-y dx + x dy)
$$
Let us proceed on the step 2. 
We apply the procedure of computing the de Rham cohomology groups 
\cite{oaku-takayama-1999}, \cite{oaku-takayama-minimal}
for $A_2/I$.
The maximal integral root of the $b$ function for 
$I =A_2 \cdot \{\ell_1, \ell_2\}$
with respect to the weight $(1,1,-1,-1)$ is $1$. 
The dehomogenization of the $(1,1,-1,-1)$-minimal filtered free resolution 
of $A_2/I$
is
\begin{equation}  \label{eq:minimal}
 A ^\bullet : \quad\quad
\xymatrix{
 A_2[0] \ar[r]^(0.4){a^{-2}} & A_2[1]\oplus A_2[0] \ar[r]^(0.7){a^{-1}} & A_2[1]
}
\end{equation}
where
\begin{eqnarray*}
   a^{-2}(c) &=& c (-\ell_2,\ell_1-1) \quad \mbox{for\ } c \in A_2 \\ 
   a^{-1}(c,d) &=& (c,d) \pmatrix{ \ell_1 \cr
                                   \ell_2 \cr}  \quad \mbox{for\ } (c,d) \in A_2[1]\oplus A_2[0] 
\end{eqnarray*}
Following \cite[procedure 1.8]{oaku-takayama-1999},
we truncate the complex
$A_2/(\partial_x A_2 + \partial_y A_2) \otimes_{A_2} A^\bullet$
to the forms of $(1,1,-1,-1)$-degree at most $1$ since the maximal integral root
of the $b$-function is $1$.
The truncated complex is
the following complex of finite dimensional vector spaces
\begin{equation}  \label{eq:truncated}
 \xymatrix{
  {\bf C}\ar[r]^(0.3){{\bar a}^{-2}} &
  ({\bf C}+{\bf C} x + {\bf C} y) \oplus {\bf C} \ar[r]^{{\bar a}^{-1}} &
  ({\bf C}+{\bf C} x + {\bf C} y) \ar[r]^(0.7){{\bar a}^0} & 0
 }
\end{equation}
Here,
\begin{eqnarray*}
 {\bar a}^{-2}(1)&=& (-\ell_2,\ell_1-1)\ {\rm mod}\, \partial_x A_2 + \partial_y A_2 \\
                 &=& (0,0) \\
 {\bar a}^{-1}(a+b x+c y,d)
                 &=& (a+b x+c y) \ell_1 + d \ell_2 \ {\rm mod}\, \partial_x A_2 + \partial_y A_2\\
                 &=& a
\end{eqnarray*}
Therefore,
the cohomology groups $H^i(A_2/(\partial_x A_2 + \partial_y A_2) \otimes A^\bullet)$
are
\begin{eqnarray*}
H^0(A_2/(\partial_x A_2 + \partial_y A_2) \otimes A^\bullet)  &=& {\rm Ker}\, \bar a^{-2} = {\bf C} \\
H^1(A_2/(\partial_x A_2 + \partial_y A_2) \otimes A^\bullet)  &=& {\rm Ker}\, \bar a^{-1} /{\rm Im}\, \bar a^{-2} 
          = ({\bf C} x + {\bf C} y)\oplus {\bf C}  \\
H^2(A_2/(\partial_x A_2 + \partial_y A_2) \otimes A^\bullet)  &=& {\rm Ker}\, \bar a^0/{\rm Im}\, \bar a^{-1} 
         = {\bf C} x + {\bf C} y
\end{eqnarray*}

Finally, we perform the step 3.
Put $\widetilde M=A_2/I$.
In order to give bases of the cohomology groups in 
$\widetilde M \otimes_R \Omega^i_R$,
we apply the transfer theorem (algorithm) of Uli Walther \cite{walther}.

We consider the following double complex (c.f., 2.4 of \cite{walther}).
$$
\xymatrix{
  \Omega(2)\otimes A_2 \ar[r]_{1\otimes a^{-2}} &\Omega(2)\otimes(A_2 \oplus A_2) \ar[r]_{1\times a^{-1}}            &\Omega(2)\otimes A_2         & \\
   A_2\otimes A_2\ar[u]\ar[r]^{\alpha^{2,-2}} &A_2\otimes(A_2\oplus A_2)\ar[u]\ar[r]^{\alpha^{2,-1}} &A_2\otimes A_2\ar[u]\ar[r]^{}& A_2\otimes \widetilde M   \\
   \AC\otimes A_2\ar[u]^{\varepsilon^{1,-2}}\ar[r]^{\alpha^{1,-2}} &\AC\otimes(A_2\oplus A_2)\ar[u]^{\varepsilon^{1,-1}}\ar[u]\ar[r]^{\alpha^{1,-1}} &\ar[u]^{\varepsilon^{1,0}}\AC\otimes A_2\ar[u]\ar[r]^{}& \AC\ar[u]^{\varepsilon^{1,-2}}\otimes \widetilde M   \\
   A_2\otimes A_2\ar[u]^{\varepsilon^{0,-2}}\ar[r]^{\alpha^{0,-2}} &A_2\otimes(A_2\oplus A_2)\ar[u]^{\varepsilon^{0,-1}}\ar[u]\ar[r]^{\alpha^{0,-1}} &A_2\ar[u]^{\varepsilon^{0,0}}\otimes A_2\ar[u]\ar[r]^{}&A_2\ar[u]^{\varepsilon^{0,-2}}\otimes \widetilde M   \\
}
$$
Here we denote $A_2/(\partial_x A_2 + \partial_y A_2)$ by $\Omega(2)$,
which is isomorphic to $\Omega^2_R$ as the right $A_2$-module. 
The vertical complex is constructed 
by the Koszul resolution of $\Omega(2)$ as the right
module denoted by $\Omega^\bullet$.
The horizontal complex is constructed by $A^\bullet$.
Note that we have the following maps in the complex:
\begin{eqnarray*}
 \varepsilon^{1,-2}((a,b)\otimes c)&=&(-\partial_y a + \partial_x b) \otimes c \\
 \varepsilon^{0,-2}(a \otimes c)&=&(\partial_x a, \partial_y a)\otimes c
\end{eqnarray*}
\begin{eqnarray*}
 \varepsilon^{1,-1}((a,b)\otimes(c,d))&=&(\partial_y a - \partial_x b)\otimes (c,d) \\
 \varepsilon^{0,-1}(a \otimes(c,d))&=&(\partial_x a, \partial_y a)\otimes (c,d)
\end{eqnarray*}
\begin{eqnarray*}
 \varepsilon^{1,0}((a,b)\otimes c)&=&(-\partial_y a + \partial_x b)\otimes c \\
 \varepsilon^{0,0}(a \otimes c)&=&(\partial_x a, \partial_y a)\otimes c
\end{eqnarray*}
\begin{eqnarray*}
 \alpha^{2,-2}(a\otimes c)&=&a \otimes c(-\ell_2,\ell_1-1) \\
 \alpha^{2,-1}(a \otimes (c,d))&=&a \otimes (c \ell_1+d\ell_2) 
\end{eqnarray*}
\begin{eqnarray*}
 \alpha^{1,-2}((a,b)\otimes c)&=& (a,b)\otimes c(-\ell_2,\ell_1-1) \\
 \alpha^{1,-1}((a,b) \otimes (c,d))&=& (a,b)\otimes(c \ell_1+d\ell_2) 
\end{eqnarray*}
\begin{eqnarray*}
 \alpha^{0,-2}(a\otimes c)&=&a \otimes c(-\ell_2,\ell_1-1) \\
 \alpha^{0,-1}(a \otimes (c,d))&=&a \otimes (c \ell_1+d\ell_2) 
\end{eqnarray*}
The last vertical complex is quasi isomorphic to $DR(\widetilde M)$.
Let us compute transfers.
Two cohomology classes $x$ and $y$ in ${\rm Ker}\, \bar a^0 \subset \Omega(2)\otimes A_2$ 
are lifted to $1\otimes x$ and $1\otimes y$ in $A_2 \otimes A_2$
respectively, and we push them to $A_2 \otimes \widetilde M$.
It follows from the definition of $\tau^2$,
$x \omega_1 \wedge \omega_2$ and $y \omega_1 \wedge \omega_2$
is the basis of $H^2(\Omega(\log f)^\bullet)$.

Let us compute transfers of bases of $H^1(\Omega(2) \otimes A^\bullet)$.
The cohomology class $1\otimes (x,0)$ in ${\rm Ker}\, \bar a^1$
are lifted to $1\otimes (x,0)$ in $A_2 \otimes (A_2 \oplus A_2)$.
We have $\alpha^{2,-1}(1 \otimes (x,0)) = 1 \otimes x \ell_1$.
Solving $-\partial_y a + \partial_x b = x \ell_1$ in $A_2$,
we obtain the preimage by $\varepsilon^{1,0}$;
we have $\varepsilon^{1,0}((-xy,x^2) \otimes 1) = x \ell_1$.
Push this element to $\AC \otimes \widetilde M$, we obtain
$(xy dx - x^2 dy) \otimes 1$.
Let us compute the preimage by $\tau^1$.
Solving $c_1 f \omega_1 + c_2 f \omega_2 = xy dx - x^2 dy$,
we obtain $c_1=0, c_2=-x$.
Therefore, $1 \otimes (x,0)$ stands for $-x \omega_2$.
Analogously,
$1 \otimes (y,0)$ is transfered to
$-y^2 dx + x y dy$ and stands for $-y \omega_2$ and
$1 \otimes (0,1)$ is transfered to
$x(y-x) dy$ and stands for $\omega_1$.
In summary,
$$H^1(\Omega(\log f)^\bullet) = {\bf C} (-x)\omega_2 + {\bf C} (-y)\omega_2 + {\bf C} \omega_1.$$

Finally, we compute transfers of bases of $H^0(\Omega(2) \otimes A^\bullet)$.
Since $\alpha^{2,-2}(1 \otimes 1) = 1 \otimes (-\ell_2,\ell_1-1)$,
we firstly need to compute the preimage of this element by
$\varepsilon^{1,-1}$.
Since the projection of this element to $\Omega(2) \otimes (A_2 \oplus A_2)$
is zero,
we have
$-\ell_2 = \partial_x x (x-y) $ and $\ell_1-1 = \partial_x x + \partial_y y$.
We decompose $1 \otimes (-\ell_2,\ell_1-1)$ as
\begin{eqnarray*}
 1 \otimes (-\ell_2,0) + 1 \otimes (0,\ell_1-1)
 &=& -\ell_2 \otimes(1,0) + (\ell_1-1) \otimes (0,1)  \\
 &=& \partial_x x(x-y) \otimes (1,0) + (\partial_x x+\partial_y y)\otimes(0,1)
\end{eqnarray*}
Since $\varepsilon^{1,-1}$ is linear, this sum is equal to
$ \varepsilon^{1,-1}(c)$ where
$c_1 = (0,-x(x-y)) \otimes (1,0) + (y,-x) \otimes (0,1)$.
Since $\alpha^{1,-1}(c_1) = (y \ell_2, -x(x-y) \ell_1-x \ell_2)\otimes 1
 = (\partial_x x y (y-x), \partial_y x y (y-x)) \otimes 1$,
the preimage of $\alpha^{1,-1}(c_1)$ by $\varepsilon^{0,0}$
is equal to $ x y (y-x) \otimes 1 \in A_2 \otimes \widetilde M$.
Therefore, the preimage of $\tau^0$ is equal to $-1$ and hence
$H^0(\Omega(\log f)^\bullet) = {\bf C}(-1)$.
Although we have done this computation by hand,
computation of transfers can be done by Gr\"obner basis computation.
See \cite{walther} and the source code for {\tt deRhamAll} of the Macaulay 2
package for D-modules \cite{leykin-tsai}.
\end{example}

Before presenting implementations and larger examples, 
we explain a bit about a procedure to find a preimage of $\tau^i$ 
in general. 
The transfer algorithm gives an element in $\Omega^i \otimes_{A_2} \widetilde M$
where $\Omega^\bullet$ is the Koszul resolution of $\Omega(2)\simeq \Omega^2_R$
as the right $A_2$-module.
This element can be identified with a differential form
with coefficients in $\tilde M$ and we need to find the preimage of it by 
$\tau^i$ which lies in $\Omega^i(\log f)$.
This can be performed by the method of undetermined coefficients.

Consider the case of $\tau^1$.
Take an element $c_1 \omega_1 + c_2 \omega_2$ in $\Omega^1(\log f)$
where $c_i \in R$.
We have seen in Theorem \ref{main} that
\begin{equation} \label{eq:tau1}
\tau^1(c_1 \omega_1 + c_2 \omega_2) =    
  f {\bar \omega_1} \otimes_{A_2} \bar c_1 + f {\bar \omega_2} \otimes_{A_2} \bar c_2 \ 
 \in \pmatrix{A_2 \cr \oplus \cr A_2 \cr} \otimes_{A_2} \widetilde M
\end{equation}
Here, we identify $\pmatrix{1\cr 0\cr} \otimes_{A_2} m_1$ with $m_1 \otimes_R dx$ and
$\pmatrix{0\cr 1\cr} \otimes_{A_2} m_2$ with $m_2 \otimes_R dy$, $m_i \in \widetilde M$
(comparison theorem)   
and  when $\omega_i = a_i dx + b_i dy$, we denote 
$\pmatrix{a_i\cr b_i\cr}$ by $\bar \omega_i$.
As the output of the transfer algorithm, we are given
an element $m_1 dx + m_2 dy$, $m_i \in \widetilde M$.
We regard $m_i$ as an element in $A_2$ in the sequel.
We rewrite $f \omega_i$ as $f \omega_1 = A dx + B dy$ and $f \omega_2 = C dx + D dy$.
Assume $I$ is generated by $\ell_1$ and $\ell_2$.
Then, the definition of $\tau^1$ (\ref{eq:tau1}) induces the following
identity in $A_2$ by taking coefficients of $dx$ and $dy$
\begin{eqnarray}
   A c_1
 + C c_2 
&=& m_1 + \sum_{j=1}^2 d_1^{j} \ell_j +\partial_x e   \label{eq:pre1tau1} \\ 
   B c_1 
 + D c_2 
&=& m_2 + \sum_{j=1}^2 d_2^{j} \ell_j +\partial_y e   \label{eq:pre2tau2}  
\end{eqnarray}
where $c_i \in R$, $d_i^{j}, e \in A_2$ are unknown.
Fix a degree bound $m$ for these elements and determine
these elements by the method of unknown coefficients.
The identities (\ref{eq:pre1tau1}) and (\ref{eq:pre2tau2})
induce a system of linear equations over ${\bf C}$ for the coefficients.
Increasing the degree bound and solving the system, 
we will be able to obtain $c_1$ and $c_2$ in finite steps by virtue of Theorem \ref{main}.

Consider the case of $\tau^2$.
Since our basis in $H^2(\Omega^\bullet \otimes \tilde M)$ is given in terms of
$x$ and $y$ and $f \omega_1 \wedge \omega_2 = dx\wedge dy$, 
we need no computation to find the preimage by $\tau^2$.

Let us consider the case of $\tau^0$.
Let $m$ be an output of the transfer algorithm. It lies in $A_2$ in general.
Finding the preimage $g$ of $\tau^0$ can be done by solving
$ g f = m + \sum_{j=1}^2 d_j \ell_j  $
where $g \in R$ and $d_j \in A_2$.

\section{Implementation and Examples} 
The second and third steps of Algorithm \ref{algorithm:main} 
can be performed with the help 
of the D-module package on Macaulay2;
use the commands {\tt DintegrationAll} to obtain the dimension of the cohomology groups, 
{\tt DintegrationClasses} to obtain the bases of cohomology groups,
and a modification of {\tt DeRhamAll} to obtain the bases of cohomology groups
in $\Omega^\bullet \otimes \widetilde M$.
Unfortunately, this implementation has not installed an efficient algorithm 
of computing $b$-function by Noro \cite{noro}
to get the truncated complex 
in \cite{oaku-takayama-1999}, \cite{oaku-takayama-2001}.
Then, only relatively small examples are feasible.
The Example \ref{example:small} is computed by our Macaulay2 program.
The Example \ref{example:big} is computed by our implementation on
kan/k0 and Risa/Asir with an implementation of \cite{noro}
(the transfer algorithm has not been implemented yet for kan/k0).
This implementation also uses the minimal filtered resolution to reduce the size
of complex of $A_2$-modules \cite{oaku-takayama-minimal}.
The program is contained in the OpenXM package with the name 
{\tt logc2.k} 
({\tt http://www.openxm.org}).
Our implementation does not contain that for the Quillen-Suslin theorem.
We utilize the implementation by A.Fabianska on Maple when the step 1-(a) fails.
We also note that computation of  the preimage of $\tau^1$ may become a bottleneck
of computation.

\begin{example} \rm  \label{example:small}
(Continued from Example \ref{example:qs} (b).)
We will determine bases of $H^i(\Omega^\bullet(\log f))$ where $f=(x^3+y^4+x y^3)(x^2+y^2)$.
We firstly use Fabianska's program for the Quillen-Suslin theorem
to find the $2$ free generators of the syzygies
of $f,f_x,f_y$.
The two rows of the following matrix $S$ are the generators 
$$
\footnotesize{ 
S=
\pmatrix{
   S_{11} &     (  -  23/6  {y}+ 1/2)   {x}^{ 2} +  (    {y}^{ 3} +  {y}^{ 2} -  2  {y})  {x}-  5/6   {y}^{ 3} &     (   1/3  {y}+ 1/2)   {x}^{ 2} +  (  -  3   {y}^{ 2} +  1/2  {y})  {x}+    {y}^{ 4} +  4/3   {y}^{ 3} -  3/2   {y}^{ 2}  \cr
    S_{21} &    -  46/75   {x}^{ 3} +  (   4/25   {y}^{ 2} -  2/25  {y})   {x}^{ 2} -   8/15   {y}^{ 2}   {x}&      4/75   {x}^{ 3} -   12/25  {y}   {x}^{ 2} +  (   4/25   {y}^{ 3} -  2/75   {y}^{ 2} )  {x}-  2/5   {y}^{ 3}  \cr
}
}
$$
where
$S_{11}=(   115/6  {y}- 5/2)  {x}   -  6   {y}^{ 3} -  43/6   {y}^{ 2} +  9  {y}$,
$S_{21}=46/15   {x}^{ 2} +  (  -  24/25   {y}^{ 2} +  22/75  {y})  {x}+  12/5   {y}^{ 2}$.
Put $A=\pmatrix{S_{12} & S_{13} \cr
                S_{22} & S_{23} \cr}$.
Then, ${\rm det}(A)=\frac{1}{3} f$.
We put $\omega_1 = \frac{1}{f} (a_{22} dx - a_{21} dy)$ and
$\omega_2 = \frac{1}{f}(-a_{12} dx + a_{11} dy)$.
($\sqrt{3}\omega_i$ agrees with the $\omega_i$ in Theorem \ref{main}.)

We apply the integration algorithm and the transfer algorithm
for $\widetilde M$.
We obtain the following result.
(1) $H^{0}(DR(\widetilde M))$ is spanned by $1\otimes f$
and then we have  $H^0(\Omega^\bullet(\log f)) \simeq {\bf C} \cdot 1$.
(2) $H^{2}(DR(\widetilde M))$ is spanned $1\otimes a$
where $a$ runs over
\begin{verbatim}
                   3        3       2     3      4
o9 = {{1}, {-x}, {y }, {-x*y }, {x*y }, {x y}, {y }}
\end{verbatim}
(We have pasted the output of our Macaulay 2 program {\tt trans.m2}.)
Then,  we have
$$H^2(\Omega^\bullet(\log f)) \simeq ({\bf C} \cdot 1 + {\bf C} \cdot (-x) 
  + \cdots + {\bf C} \cdot y^4 ) \, \omega_1 \wedge \omega_2
$$
(3) $H^{1}(DR(\widetilde M))$ is spanned by $3$ differential forms
$ m_1 dx + m_2 dy$ where $m_1$, $m_2$ are elements in $A_2$,
of which explicit expressions are a little lengthy. 
We solve the identities (\ref{eq:pre1tau1}) and (\ref{eq:pre2tau2})
to find $c_1$ and $c_2$. 
In other words, we need to compute  preimages of $m_1 dx + m_2 dy$  by $\tau^1$.
As we explained, this can be done by the method of undetermined coefficients degree by degree.
We can find solutions when the degree of $c_i, d^j_i, e$
with respect to $x,y$ is $6$ and that with respect to $\partial_x, \partial_y$
is $0$.
Here is a basis of $3$-dimensional vector space $H^1(\Omega^\bullet(\log f))$ obtained by this method.
\begin{eqnarray*}
&& -y x \omega_1 -4/25 x^2 \omega_2 \\
&& ((   215/28  {y}- 1101/280)  {x}-  367/56   {y}^{ 2} ) \omega_1 
  +(   43/35   {x}^{ 2} -   367/350  {y}  {x}  ) \omega_2 \\
&& ((  {y}- 11/30)  {x}   -  28/9   {y}^{ 3} -  13/6   {y}^{ 2} +  14/3  {y}) \omega_1 
+  ( 4/25   {x}^{ 2} +  (  -  112/225   {y}^{ 2} +  2/5  {y})  {x}+  56/45   {y}^{ 2} ) \omega_2  
\end{eqnarray*}
All programs and session logs to find this answer is obtainable from \\
{\tt http://www.math.kobe-u.ac.jp/OpenXM/Math/LogCohomology/2007-11/log-2007-11-22.txt}
The logarithmic comparison theorem does not hold for this example.
In fact, the dimensions of the de Rham cohomology groups $H^i(\Omega_f^\bullet)$,
$(i=2,1,0)$ are $5,3,1$ respectively.
\end{example}

\begin{example} \rm  \label{example:big}
We apply a part of our algorithm to
compute the dimensions of the cohomology groups $H^i(\Omega^\bullet(\log f))$
for $f=x^p+y^q+x y^{q-1}$.
Here is a table of $p$, $q$ and the dimensions of
$H^2$, $H^1$, $H^0$  and timing data. \\
\begin{tabular}{|c|c|l|l|}
\hline
$p$ & $q$ & Dimensions & Timing in seconds \\ \hline \hline
$10$ & $11$ & (8,1,1)  &  3.5 \\ \hline
$10$ & $12$ & (9,1,1)  &  4.6 \\ \hline
$10$ & $13$ & (10,1,1)  &  6.9 \\ \hline
$10$ & $14$ & (11,1,1)  &  9.4 \\ \hline
$10$ & $20$ & (17,1,1)  &  55.0 \\ \hline
$10$ & $21$ & (18,1,1)  &  86.8 \\ \hline
\end{tabular} \\
The program is executed on a machine with 2G RAM
and Pentium III (1G Hz).  
\end{example}
The homogenization of $f$, $f_x$, $f_y$ generates an ideal that is Cohen-Macaulay.
These examples do not need to call the subprocedure Quillen-Suslin.
However, the logarithmic comparison theorem does not hold for these examples.
Computation of de Rham cohomology groups is not feasible by our implementation.

\section{A Yet Another Algorithm}
In the previous section, we have presented a general algorithm
of computing a basis
of the logarithmic cohomology groups for plane curves.
However, this algorithm relies on algorithms for the Quillen-Suslin theorem
and they are sometimes slow.
We will present a yet another algorithm, which is free from the Quillen-Suslin
theorem, but it works only for computing a basis of 
the middle dimensional cohomology
group $H^2(\Omega^\bullet(\log f))$ under some conditions on $f$.
This section can be read independently from other sections.
For reader's convenience, we will also redefine some notations.

Before stating the main algorithm,
we start with an introductory example, which explains
the idea of our algorithm.

Put $K={\bf C}$ and $L= (1-x)x \partial + 2x (= \theta_x - x (\theta_x-2))$.
We consider the problem of 
determining a basis of the $K$-vector space  $K[x]/L \cdot K[x]$.
Since $L$ is a $K$-linear map and
$K[x]$ is an infinite dimensional $K$-vector space,
the quotient has the structure of a $K$-vector space.
However, note that $L \cdot K[x]$ is not an ideal 
and we cannot use Gr\"obner basis to get a basis.

Let us act $L$ on monomials;
$L \cdot x^k = k x^k - (k-2) x^{k+1}$.
For small $k$, they are
$ L \cdot 1 = 2 x $,
$ L \cdot x = x + x^2 $,
$ L \cdot x^2 = 2 x^2 $.
Then $ x^{k+1} \simeq \frac{k}{k-2} x^k$ modulo $L \cdot K[x]$.
In particular, if $k \geq 3$, then the monomial $x^{k+1}$ can be reduced 
to a lower order monomial
modulo $L \cdot K[x]$.
Hence, the set of monomials $1, x, x^2, x^3$ generates $ K[x]/L \cdot K[x]$.
More precisely, we can prove that it is isomorphic to $F_3/L \cdot F_2 $.
Where $F_k$ is the set of polynomials of which degree is
less than or equal to $k$.
The monomials $1, x, x^2, x^3$ are not independent modulo $L \cdot K[x]$
and satisfies the relation above.
Finally, we conclude that  $K[x]/L K[x] \simeq K \cdot 1 + K \cdot x^3$. 

Note  that $3$ is the magic number,
which is characterized as follows.
Put $L^* = -(1-x)x \partial -1 + 4x$.
$ {\rm in}_{(1,-1)} ( L^* )  \cap K[-\partial x] $
is generated by $b(-\partial x)$ where
$b(s) = s-3$.
The polynomial $b(s)$ is called the indicial polynomial 
($b$-function) for integration.
The magic number $3$ is the root of $b(s)=0$.
We will call the method to bound a degree by a root of a $b$-function
{\it $b$-function criterion}.
T.Oaku firstly introduced the $b$-function criterion to compute
restrictions and integrations of $D$-modules \cite{Ob}.
The topic of computing $K[x]/L \cdot K[x]$ by the $b$-function
was also discussed in more detail in an expository book 
``D-modules and Computational Mathematics'' (in Japanese)
by T.Oaku. 

Let $f$ be a polynomial in two variables. 
Put
$$ \Omega^k_f = \mbox{$k$-form with coefficients in $K[x,y,1/f]$}$$
As we have explained in the introduction,
the $k$ form
$\omega \in \Omega_f^k$ is called {\it logarithmic} $k$-form
iff both of $f \omega$ and $d f \wedge \omega$ have polynomial
coefficients.
The space of logarithmic $k$-forms is denoted by
$\Omega^k(\log f)$.
The question we address in this section is the computation of
$ \frac{\Omega^2(\log f)}{d \, \Omega^1(\log f)} $.
It is easy to see that 
$\displaystyle{\Omega^2(\log f) = \frac{K[x,y]dx\wedge dy}{f}} $.
Let us determine all the logarithmic $1$-forms.
Let  $(p,q,r)$ a triple of polynomials such that
\begin{equation}
f_y p - f_x q + f r =0   \quad \mbox{(syzygy equation).}
\end{equation}
Note that $(0,f,f_x)$, $(f,0,-f_y)$, $(f_x,f_y,0)$ are 
trivial solutions of the syzygy equation.
For a solution $(p,q,r)$ of the syzygy equation,
$\omega=\frac{p dx + q dy}{f}$ belongs to $\Omega^1(\log f)$.
Conversely, any logarithmic $1$-form can be expressed in this way.
In fact, the condition that $df \wedge \omega$ has a polynomial
coefficient is equivalent to that $f_y p - f_x q$ is a multiple of $f$.

Put $\omega=\frac{p dx + q dy}{f}$.
Let $e(x,y)$ be any polynomial.
Then, 
$ d(e \omega) = (Le) \frac{dx\wedge dy}{f} $
where
$$ L= q \partial_x -p \partial_y + q_x - p_y + \frac{f_y p - f_x q}{f}$$

We denote the Weyl algebra $A_2$ by $D$ for simplicity in the sequel.
Suppose that
$L_i$, $(i=1, \ldots, m)$  stand for a set of generators of the solution space
of the syzygy equation, which is a $K[x,y]$-module.
Then $d\, \Omega^1(\log f) = \sum L_i K[x,y]  dx\wedge dy/f$.
Therefore, the computation of $H^2$ is nothing but the computation
of $K[x,y]/\sum_{i=1}^m L_i \bullet K[x,y]$.
Put $I^* = D\cdot \{ L_1^*, \ldots, L_m^* \}$, which is a left $D$ ideal.
We denote by $F_k$ the $K$-subvector space of $D$ of which
$(1,1,-1,-1)$-order is less than or equal to $k$
\cite[p.14, p.203]{SST}

\noindent
\begin{algorithm} \rm \label{algorithm:logch2}
$H^2(\Omega^\cdot(\log f))$. \\
Step 1.  Find generators of the syzygy equation and obtain
explicit expressions of $L_i$. \\
Step 2. Compute $(1,1,-1,-1)$-Gr\"obner basis (standard basis) of $I$.
We denote the elements of the Gr\"obner basis by ${L^i}^*$ (renaming). \\
Step 3.  Find the monic generator $b(-\partial_x x - \partial_y y)$ of
$ {\rm in}_{(1,1,-1,-1)}\,(I) \cap K[-\partial_x x - \partial_y y]$. \\
Step 4. Let $k_0$ be the maximal non-negative root of $b(s)=0$.
Then,  return $K$-vector space basis $\{ c_i \}$ of
$$ F_{k_0}/\sum_i L_i \cdot F_{k_0 - {\rm ord}_{(1,1,-1,-1)}\,(L_i)}.  $$ 
$\{ c_i dx \wedge dy / f \}$ is a basis of $H^2$.
\end{algorithm}

\noindent
The steps 2, 3, 4 can also be done by computing
$ D/(I^* + \partial_x D + \partial_y D)$ ($0$-th integral module)
where $I^*$ is the formal adjoint of $I$.
(As to details for the steps 2, 3, 4, see \cite{OTT}.)

\noindent
{\bf Note}:
Although our discussion is independent from the discussions of the previous sections,
the left ideal generated by $L_i^*$ is nothing but
$\widetilde{Der}_R(-\log f)$
and hence this algorithm and the Algorithm \ref{algorithm:main}
are analogous for computing a basis of $H^2(\Omega^\bullet(f))$.
We also note that finding bases for 
$H^i(\Omega^\bullet \otimes_{A_2} \widetilde M)$ can be performed
by applying the integration algorithm and the transfer algorithm for 
$D/I^*$.
The Algorithm \ref{algorithm:main} relies on algorithms for the Quillen-Suslin
theorem to find bases for $H^i(\Omega(\log f)^\bullet)$, $i=1,0$.

\begin{theorem}
If 
 ${\rm dim}\,V(f,f_x, f_y)\leq 0$,
 ${\rm dim}\,V(f,f_x)\leq 1$,
 ${\rm dim}\,V(f,f_y)\leq 1$,
 then the Algorithm \ref{algorithm:logch2}
is correct.
\end{theorem}

\begin{proof}
Let $I$ be the left ideal in $D$ generated by $L_1, \ldots, L_m$.
We may assume that $I$ contains $f \partial_x$, $f \partial_y$ and
$f_y \partial_x - f_x \partial_y$.
Therefore, the characteristic variety of $I$ is contained in
$V(f(x,y)\xi, f(x,y)\eta, f_y(x,y) \xi - f_x(x,y) \eta)$, of which dimension
is less than or equal to $2$ from the assumption.
In fact, assume $(a,b) \in V(f,f_x,f_y)$.
Then, $\xi$ and $\eta$ are free and then the dimension of the characteristic variety
is less than or equal to $2$.
Assume $(a,b) \in V(f,f_x) \setminus V(f,f_x,f_y)$.
Then, we have $f(a,b)=0, f_x(a,b)=0$ and $f_y(a,b) \not=0$.
Then, $\eta$ is free and $\xi = 0$ and then the dimension of the characteristic variety
is less than or equal to $2$.  
The rest cases can be shown analogously.
Therefore, $D/I$ is a holonomic $D$-module 
and hence a non-trivial $b$ exists (\cite[Chapter 5, Theorem 5.1.2]{SST}).
The rest of the correctness proof is analogous with
that of the $0$-th integration algorithm of $D$-modules
\cite{Ob}, \cite[Chapter 5; Theorems 5.2.6 and 5.5.1]{SST}.
\end{proof}
\bigbreak

\noindent
{\bf Note}: The algorithm works to get $H^n(\Omega^\bullet(\log f))$
in the $n$-variable case if ${\rm dim}\, V(f,f_{x_{i_1}}, \ldots, f_{x_{i_m}}) \leq n-m$
for all $m=1, \ldots, n$ and all combinations $i_1, \ldots, i_m$.
The algorithm and the correctness proof are analogous.
In fact,
since $f \xi_i$ and $ (-1)^i f_{x_j} \xi_i - (-1)^j f_{x_i} \xi_j $, 
$( 1\leq i \not= j \leq n)$ are in the characteristic ideal for $I$ 
and then the dimension 
of the characteristic variety is less than or equal to $n$
by utilizing the condition.


\begin{example} \rm
For $f=(x^3+y^4+x y^3)(x^2+y^2)$, we have
${\rm dim}\, H^2(\Omega^\bullet(\log f)) = 7$
with our yet another algorithm \ref{algorithm:logch2}.
The execution time is 1.9s.
We need to call the procedure Quillen-Suslin if we use the first algorithm. 
\end{example} 

\medbreak

We close this paper with
a final note and the acknowledgement of this paper.
We think that logarithmic differential forms give nice simple bases
for some of hypergeometric integrals as pairings of twisted cycles and cocycles
when the logarithmic comparison theorem holds for twisted de Rham complex.  
We hope that our result have applications to study hypergeometric integrals.
The authors are grateful to A.Fabianska for helping us to compute free bases
of syzygies by using her implementation for Quillen-Suslin's theorem.

\end{document}